\documentclass[reqno,12pt,a4paper]{amsart}

\usepackage[shortlabels]{enumitem}
\usepackage{amstext,amsmath,amsthm,amsfonts,amssymb,amscd}
\usepackage{latexsym,mathrsfs,dsfont,euscript}

\usepackage{marginnote}
\usepackage{geometry} 
\usepackage{fullpage}

\usepackage{multicol}

\usepackage{xcolor}

\usepackage{hyperref}
\hypersetup{
  colorlinks,%
  citecolor=blue,%
    filecolor=red,%
    linkcolor=red,%
    urlcolor=blue
}

\newtheorem{theorem}{Theorem}[section]
\newtheorem{proposition}[theorem]{Proposition}
\newtheorem{lemma}[theorem]{Lemma}

\theoremstyle{definition}
\newtheorem{definition}[theorem]{Definition}

\theoremstyle{remark}

\numberwithin{equation}{section}

\newcommand{\abs}[1]{\left\vert#1\right\vert}
\newcommand{\set}[1]{\left\{#1\right\}}

\newcommand{\norm}[1]{\left\Vert#1\right\Vert}

\allowdisplaybreaks

\newcommand{\R}{\mathbb{R}}
\newcommand{\Z}{\mathbb{Z}}

\newcommand{\Hh}{\mathscr{H}}
\newcommand{\D}{\mathscr{D}}
\title[]{On the structure of the diffusion distance induced by the fractional dyadic Laplacian}
\author[]{Mar\'ia Florencia Acosta}
\author[]{Hugo Aimar}
\author[]{Ivana G\'{o}mez}
\author[]{Federico Morana}
\subjclass[2010]{Primary 54E35, 35K08}
\keywords{Diffusion metrics; Dyadic diffusion}
\begin{document}

\maketitle

\begin{abstract}
	In this note we explore the structure of the diffusion metric of Coifman-Lafon determined by fractional dyadic Laplacians. The main result is that, for each ${t>0}$, the diffusion metric 
	is a function of the dyadic distance, given in $\R^+$ by $\delta(x,y) = \inf\set{\abs{I}\colon I \text{ is a dyadic interval containing } x \text{ and } y}$. Even if these functions of $\delta$ are not equivalent to $\delta$, the families of balls are the same, to wit, the dyadic intervals.
\end{abstract}

\section{Introduction}
Let $W_t(x)$ be the Weierstrass kernel in $\R^n$. The gaussian decay of $W_t(x)$ as a function of $x$ for $t>0$ fixed guarantees the convergence of the integral
$$d_t^2(x,y)=\int_{z\in\R^n}\big|W_t(x-z)-W_t(y-z)\big|^2\,dz.$$
Moreover, $d_t$ is a metric on $\R^n$ for each $t>0$. It is clear that different values of $t$ will produce a diversity of metrics. Nevertheless the family of all $d_{t_1}$-balls is the same as the family of the $d_{t_2}$-balls for any choice of $t_1$ and $t_2$. Furthermore, this unique class of balls coincide with the class of all Euclidean balls. Let us precise these remarks in the following statement.

\begin{proposition}\label{propo1}
	Let $d_t$ be defined as before. Then\begin{enumerate}[(a)]
		\item $d_t$ is translation invariant;
		\item $d_t(x,y)$ depends only on $\abs{x-y}$, i.e. $d_t(x,y)=\rho_t\left(\abs{x-y}\right)$;
		\item $\rho_t$ is strictly increasing and continuous, with $\rho_t(0)=0$;
		\item $\lim_{r\to 0^+}\frac{\rho_{t_1}(r)}{\rho_{t_2}(r)}=\bigl(\frac{t_2}{t_1}\bigr)^{\tfrac{n}{4}+\tfrac{1}{2}}$, $t_1$, $t_2>0$;
		\item the family of $d_t$-balls are the Euclidean balls.
	\end{enumerate}
\end{proposition}
\begin{proof}
	To prove $(a)$ we only have to change variables $y-z=u$ in the integral defining $d_t^2(x,y)$. Hence $d_t(x,y)=d_t(x-y,0)$. To prove $(b)$ we have to show that the function of $x$, $d_t(x,0)$, is rotation invariant. Take a rotation $R$ of $\R^n$. Then, since $W_t$ is radial
	\begin{align*}
		d_t^2(Rx,0)&=\int_{z\in\R^n}\big|W_t(Rx-z)-W_t(z)\big|^2\,dz\\
		&=\int_{z\in\R^n}\big|W_t(R(x-R^{-1}z))-W_t(R^{-1}z)\big|^2\,dz\\
		&=\int_{u\in\R^n}\big|W_t(R(x-u))-W_t(u)\big|^2\,du\\
		&=d_t^2(x,0).
	\end{align*}
	In order to prove $(c)$, notice that since $d_t^2$ is a radial function of $x-y$, a formula for the profile $\rho_t^2(r)$ is given for $r>0$ by $\rho_t^2(r) = \int_{z\in\R^n}\big|W_t(r\vec{e_1}-z)-W_t(z)\big|^2\,dz$, with $\vec{e_1}$ the first vector of the canonical basis or $\R^n$ and $W_t(y)=(4\pi t)^{-n/2}e^{-\abs{y}^2/4t}$. For $t$ fixed, the derivative of $\rho_t^2$ as a function of $r>0$ is given by
\begin{align*}
	\frac{d\rho_t^2}{dr}(r) &= \frac{1}{(4\pi t)^n} \int\limits_{z\in\mathbb{R}^n} 2\left(e^{-\tfrac{\abs{r\vec{e}_1-z}^2}{4t}}-e^{-\tfrac{\abs{z}^2}{4t}}\right) \frac{(-2)(r-z_1)}{4t}\,e^{-\tfrac{\abs{r\vec{e}_1-z}^2}{4t}} \,dz\\
	&= \frac{-4}{(4\pi t)^n} \left[ \int\limits_{z\in\mathbb{R}^n}e^{-\frac{2\abs{r\vec{e}_1-z}^2}{4t}}\,\frac{(r-z_1)}{4t}\,dz - \int\limits_{z\in\mathbb{R}^n}e^{-\frac{\abs{z}^2}{4t}} e^{-\frac{\abs{r\vec{e}_1-z}^2}{4t}}\,\frac{(r-z_1)}{4t}\,dz \right]\\
	&= -2 \int\limits_{z\in\mathbb{R}^n} \frac{e^{-\tfrac{\abs{z}^2}{4t}}}{(4\pi t)^{-n/2}} \frac{e^{-\tfrac{\abs{r\vec{e}_1-z}^2}{4t}}}{(4\pi t)^{-n/2}}(-2)\frac{(r-z_1)}{4t}\,dz\\
	&= -2\left(W_t\ast\frac{\partial W_t}{\partial x_1}\right)(r\vec{e}_1)\\
	&= -2\,\frac{\partial}{\partial x_1}\big(W_t\ast W_t\big)(r\vec{e}_1)\\
	&= -2\,\frac{\partial}{\partial x_1}W_{2t}(r\vec{e}_1)\\
	&= -2\,\frac{\partial}{\partial x_1} \left(\frac{1}{(8\pi t)^{n/2}}e^{-\tfrac{\abs{x}^2}{8t}}\right)(r\vec{e}_1)\\
	&= \left(\frac{4}{(8\pi t)^{n/2}} e^{-\tfrac{\abs{x}^2}{8t}}\frac{x_1}{8t}\right)(r\vec{e}_1)\\
	&= \frac{4}{(8\pi t)^{n/2}}\,e^{-\tfrac{r^2}{8t}}\,\frac{r}{8t}\\
	&> 0\,,
\end{align*}
	where we have used the semigroup property of $W_t$, and \textit{(c)} is proved. Property \textit{(e)} is now a consequence of \textit{(c)}.	Let us finally check \textit{(d)}. The above computation of $\tfrac{d\rho_t^2}{dr}$ and l'H\^{o}pital rule gives
	\begin{equation*}
		\lim_{r\to 0^+}\frac{\rho_{t_1}^2(r)}{\rho_{t_2}^2(r)}=\lim_{r\to 0^+}\left(\frac{t_2}{t_1}\right)^{\tfrac{n}{2} +1} e^{\tfrac{r^2(t_1-t_2)}{8t_1 t_2}}=\left(\frac{t_2}{t_1}\right)^{\tfrac{n}{2}+1}
	\end{equation*}
and \textit{(d)} is proved.
\end{proof}

In more general structures where no geometric invariances are available the question of the structure of diffusive metrics becomes less simple and more interesting. In particular it could be important, after the lost of equivalence for different values of time $t$ of the associated metrics, to consider the stability of the family of balls. In other words, for $t_1\neq t_2$, is it true that for every $x$ and every $r_1>0$ there exists a positive $r_2$ such that $B_{d_{t_1}}(x,r_1)=B_{d_{t_2}}(x,r_2)$?

We solve this problem for the non-convolution case determined by the dyadic fractional diffusion whose spectral analysis is supplied by the Haar wavelets. For the sake of simplicity we shall work in $\R^+$, even when the whole analysis can be carried over quadrants in $\R^n$ and even over much more general 
structures with dyadic analysis.

In Section~2 we introduce the basic definitions and notation. Section~3 is devoted to state and prove our main result.

\section{Definitions and notation}
Let $\D$ be the family of all dyadic intervals in $\R^+=\set{x\geq0}$. Precisely $\D = \big\{I^j_k=[k2^{-j},(k+1)2^{-j})\colon j\in\mathbb{Z}, k\in\mathbb{N}_0\big\}$. With $\D^j = \big\{I^j_k\colon k\in\mathbb{N}_0\big\}$ we have that $\D=\bigcup_{j\in\Z}\D^j$.

Let $\Hh=\big\{h^j_k=2^{j/2}h^0_0(2^jx-k)\colon j\in\mathbb{Z},k\in\mathbb{N}_0 \big\}$ be the Haar wavelet system 
in $\mathbb{R}^+$, with $h^0_0=\chi_{[0,\frac{1}{2})}(x)-\chi_{[\frac{1}{2},1)}(x)$, where as usual $\chi_E$ denotes the indicator function of the set $E$.
The family $\Hh$ constitutes an orthonormal basis of $L^2\left(\mathbb{R}^+\right)$. 
Let $h_I$ denotes the Haar wavelet supported on the dyadic interval $I$, so for $I=I^j_k$ we have that $h_I=h^j_k$.
Let $I(h)$ denotes the dyadic interval that supports the wavelet $h\in\mathscr H$.

\begin{definition}
	The dyadic distance is defined by
	\begin{equation*}
		\delta(x,y) = \inf\set{\abs{I}\colon I \text{ is a dyadic interval containing } x \text{ and } y}.
	\end{equation*}
\end{definition}
Notice that if $x\neq y$ there exists a smallest dyadic interval containing $x$ and $y$, which we will denote by $I(x,y)$. Taking $I(x,x)=\set{x}$, we have that $\delta(x,y)=|I(x,y)|$ for every $x,y\in\R^+$.

The metric $\delta$ on $\R^+$ is not translation invariant and is an upper bound for the Euclidean. In fact $|x-y|\leq\delta(x,y)$. Of course, they are not equivalent. This means that $\delta(x,y)$ is in general much larger than $|x-y|$. Hence we could expect some better integrability properties of the powers of $\delta(x,y)$, 
locally and/or globally. Nevertheless, the behavior of the local and global integral properties of $\delta(x,y)$ are exactly the same as those of the powers of $|x-y|$. From a general point of view these properties are consequences of the fact that $\big(\R^+,\delta,m\big)$ is a normal or 1-Ahlfors regular space of homogeneous type (see \cite{MaSe79}) without atoms and with infinite total Lebesgue measure $m$. Then the integrals of $\delta^\alpha(x,y)$, $\alpha\in\R$, inside $B_\delta(x,r)$ and outside $B_\delta(x,r)$, for $r>0$, are exactly the same as the integrals of $|x-y|^\alpha$ inside and outside the corresponding Euclidean balls $(x-r,x+r)$. In particular, the local and global singularity is provided by $\delta(x,y)^{-1}=\frac{1}{\delta(x,y)}$. Hence, the natural fractional integrals or Riesz type operators of the setting are given by kernels of the form $\delta(x,y)^{-1+s}=\frac{1}{\delta(x,y)^{1-s}}$ for $s>0$. So that the natural fractional differential operators are defined by kernels of the form $\delta(x,y)^{-1-s}=\frac{1}{\delta(x,y)^{1+s}}$ for $0<s<1$. Of course, as in the Euclidean case, the strong local singularity of this kernel needs for some regularity of the functions in the domain of the operator. As proved in \cite{AiGo18} the indicator function of a dyadic interval belongs to the class of Lipschitz-1 functions with respect to $\delta$. In particular, the Haar wavelets in $\Hh$ are all smooth in this sense. Actually, for $f$ bounded and Lipschitz-$\sigma$ for $0<s<\sigma\leq1$ we have that
$$D^s_{dy}f(x)=\int_{\R^+}\frac{f(y)-f(x)}{\delta(x,y)^{1+s}}\,dy$$
is well defined. We call $D^s_{dy}f$ the dyadic fractional Laplacian of $f$ in $\R^+$.

The initial value problem
\begin{equation*}\label{pvi.dyfracdif}
	\left\{  \begin{array}{rl}
		\dfrac{\partial u}{\partial t}(x,t) \!\!\!&= D^s_{dy} u(x,t)\vspace{4pt}\\
		u(x,0) \!\!\!&= f(x)
	\end{array} \right.
\end{equation*}
was considered in \cite{AcAimCzech16} 
and, like in the Euclidean case, can be solved as an integral operator, which of course lacks the convolution structure. In fact
\begin{equation*}
	u(x,t) = \int_{\R^+} K_s(x,y;t)f(y)\,dy,
\end{equation*}
where
\begin{equation*}
	K_s(x,y;t) = \sum_{h\in\mathscr{H}} e^{-t|I(h)|^{-s}}h(x)h(y).
\end{equation*}

Following \cite{CoifmanLafon06} we may try to define a dyadic fractional diffusion type distance.
\begin{definition}
	For $t>0$ and $s>0$, the fractional dyadic diffusion distance of order $s$ at time $t$ is given by
	$$d_{t}(x,y)= \sqrt{\int_{z\in\R^+}\left|K_s(x,z;t)-K_s(y,z;t)\right|^2\,dy}.$$
\end{definition}

In the next results we will explore the analogous features of $d_t$ to those of the Euclidean case stated in Proposition~\ref{propo1}. 
First, we shall see the good definition and the metric character of $d_t$, and determine its spectral representation through the Haar wavelet system.

\begin{proposition}\label{propo2}
	Let $s>0$ and $t>0$ be given. Then $d_t$ is well defined, is a metric on $\R^+$ and can be computed as
	$$d_{t}(x,y)= \sqrt{\sum_{h\in\mathscr H}e^{-2t|I(h)|^{-s}}\left|h(x)-h(y)\right|^2}.$$
\end{proposition}
\begin{proof}
	First, notice that the diffusion kernel $K_s(x,y;t)$ is well defined and finite for every $x,y\in\R^+$. Indeed, as $|h_I(w)|=|I|^{-\frac{1}{2}}\chi_{I}(w)$ 
	so $K_s(x,y;t)=\sum_{I\supseteq I(x,y)} e^{-t|I|^{-s}}h_I(x)h_I(y)$ whose absolute series is bounded above by
	$\sum_{I\supseteq I(x,y)} |I|^{-1} = \sum_{j\in\mathbb{N}_0} 2^{-j}|I(x,y)|^{-1} = 2\,|I(x,y)|^{-1} = 2\delta(x,y)^{-1}$. By definition, $d_t$ is the norm of the difference of the diffusion kernel at time $t$ centered at two points in consideration, so the metric properties follow trivially. As well, by Parseval's identity
	\begin{align*}
		d_t(x,y)^2 &= \norm{K_s(x,\cdot;t)-K_s(y,\cdot;t)}^2\\
		&= \norm{\sum_{h\in\mathscr H} e^{-t|I(h)|^{-s}}\left[h(x)-h(y)\right]h}^2\\
		&= \sum_{h\in\mathscr H}e^{-2t|I(h)|^{-s}}\left|h(x)-h(y)\right|^2.
	\end{align*}
	The finiteness of $d_t(x,y)$ will follow from the next results.
\end{proof}

\section{Main results}

\begin{lemma}\label{lemma1}
	Let $s>0$. For $t>0$ define $\psi_t(\lambda)= \sqrt{\frac{2}{\lambda}\eta_t\left(\lambda^{-s}\right)}$ with $\eta_t(\sigma)=2e^{-2t\sigma}+\sum_{\ell\geq1}2^\ell e^{-2t2^{s\ell}\sigma}$. 
	Then, when restricted to the sequence of integer powers of 2, $\{2^j\colon j\in\mathbb{Z}\}$, we have that
	\begin{enumerate}[label=(\alph*)]
		\item $\psi_t$ is strictly increasing;
		\item $\psi_t(0^+)=0$;
		\item ${\psi_t(+\infty) \simeq t^{-\frac{1}{2s}}}$.
	\end{enumerate}
\end{lemma}
\begin{proof}
	Define, for $i\in\mathbb{Z}$,
	\begin{align*}
		f(i)&:= \frac{1}{2}\,\psi_t^2(2^i)\\
		&= 2^{1-i}e^{-2t2^{-is}} + \sum_{\ell\geq1} 2^{\ell-i} e^{-2t2^{-is}2^{\ell s}}\\
		&= 2^{1-i}e^{-2t2^{-is}} + \sum_{k=\ell-i\geq1-i} 2^{k} e^{-2t2^{ks}}.
	\end{align*}
	
	Then $$f(i+1)= 2^{-i}e^{-2t2^{-(i+1)s}} + \sum_{k\geq-i} 2^{k} e^{-2t2^{ks}}$$
	and so
	\begin{align*}
		f(i+1)-f(i)&= 2^{-i}e^{-2t2^{-(i+1)s}} - 2.2^{-i}e^{-2t2^{-is}} + 2^{-i} e^{-2t2^{-is}}\\
		&= 2^{-i}e^{-2t2^{-is}2^{-s}} - 2^{-i}e^{-2t2^{-is}}\\
		&= 2^{-i} \left[ \xi^{2^{-s}} - \xi \right]\\
		&>0
	\end{align*}
	because the function $\xi^x$ is monotone decreasing in the variable $x$ (since $\xi:=e^{-2t2^{-is}}$ is positive and less than one) and $2^{-s}<1$. 
	This shows that $\psi_t^2$ is an increasing function and therefore so is $\psi_t$, on account of its positivity. Thus $(a)$ is demonstrated.

	To check $(b)$ notice that
	$$\lim_{i\to-\infty} f(i) = \lim_{i\to-\infty} 2^{1-i}e^{-2t2^{-is}}
		= 2 \lim_{x\to+\infty} xe^{-2tx^{s}}
		= 0$$
	and  so $$\lim\limits_{i\to-\infty} \psi_t(2^i) = 0.$$
	
	In order to prove $(c)$ notice first that
	\begin{align*}
		\lim_{i\to+\infty} f(i) &= \lim_{i\to+\infty} 2^{1-i}e^{-2t2^{-is}} + \sum_{k\in\mathbb{Z}} 2^{k} e^{-2t2^{ks}}\\
		&= \sum_{k\in\mathbb{Z}} 2^{k} e^{-2t2^{ks}}\\
		&< \sum_{k\in\mathbb{Z}} 2\int_{2^{k-1}}^{2^k}e^{-2tx^{s}}\,dx\\
		&= 2\int_{0}^{+\infty}e^{-2tx^{s}}\,dx\\
		&< +\infty
	\end{align*}
	which implies that $\,\lim_{i\to+\infty} \psi_t(2^i)= \psi_t(+\infty)< +\infty$. 
	On the other hand,
	we attain a lower bound from
	$$\sum_{k\in\mathbb{Z}} 2^{k} e^{-2t2^{ks}}> \sum_{k\in\mathbb{Z}} \int_{2^{k-1}}^{2^k}e^{-2tx^{s}}\,dx= \int_{0}^{+\infty}e^{-2tx^{s}}\,dx.$$
	So, since $\psi_t(+\infty) = \sqrt{2\lim_{i\to+\infty} f(i)}\,$, we have that
	\begin{equation*}
		\sqrt{2}\;c_{t}(s) < \psi_t(+\infty) < 2\,c_{t}(s)
	\end{equation*}
	for $\,c_{t}(s) = \sqrt{\int_{0}^{+\infty}e^{-2tx^{s}}\,dx} = t^{-\frac{1}{2s}} \sqrt{\int_{0}^{+\infty}e^{-2x^{s}}\,dx}$.
\end{proof}
At this point it is important to remark that, in contrast with property \textit{(d)} in Proposition~\ref{propo1}, now for $0<t_1<t_2$ we have that $d_{t_2}(x,y)\leq d_{t_1}(x,y)$ for every $x$, $y\in\mathbb{R}^+$. On the other hand, there is no constant $C>0$ such that the inequality $d_{t_1}(x,y)\leq C d_{t_2}(x,y)$ holds for every $x$, $y\in\mathbb{R}^+$. In fact, both observations above follow from the fact that
\begin{equation*}
	\frac{d_{t_2}^2(x,y)}{d_{t_1}^2(x,y)}\leq e^{-2(t_2-t_1)\delta^{-s}(x,y)}
\end{equation*}
for every $x$ and $y$  in $\mathbb{R}^+$.

From Lemma~\ref{lemma1} we can deduce that the graph of $\psi_t$ is flatter as $t$ increases and, conversely, it reaches higher values at infinity as $t$ approaches zero. 

\begin{theorem}
	Let $d_t$ be the fractional dyadic diffusion metric of order $s>0$ at $t>0$.
	Let $\psi_t$ be as in Lemma~\ref{lemma1}, with $\psi_t(0):=0$. Then
	\begin{enumerate}
		\item $d_{t}(x,y)=\psi_t(\delta(x,y))$\label{eq:propo2} for 
		$x,y\in\R^+$;
		\item the family of $d_t$-balls, given as usual by $B_t(x,r)=\set{y\in\R^+\colon d_t(x,y)<r}$ for $x\in\R^+$ and $r>0$, coincides with $\D$, the family of all dyadic intervals.
	\end{enumerate}
\end{theorem}

\begin{proof}
	In order to prove $(1)$, let us use the representation formula for $d_t^2$ provided by Proposition~\ref{propo2}. For $x\neq y$,
	\begin{align*}
	d_{t}^2(x,y) &= \sum_{h\colon x\in I(h) \vee y\in I(h)} e^{-2t|I(h)|^{-s}}\left|h(x)-h(y)\right|^2\\
	&= e^{-2t|I(x,y)|^{-s}} \left|h_{I(x,y)}(x)-h_{I(x,y)}(y)\right|^2\\ &\qquad + \sum_{h\colon x\in I(h) \wedge y\notin I(h)} e^{-2t|I(h)|^{-s}}\left|h(x)\right|^2\\ &\qquad + \sum_{h\colon x\notin I(h) \wedge y\in I(h)} e^{-2t|I(h)|^{-s}}\left|h(y)\right|^2\\
	&= 4|I(x,y)|^{-1}e^{-2t|I(x,y)|^{-s}} + 2\sum_{\ell\geq1} e^{-2t\left(2^{-\ell}|I(x,y)|\right)^{-s}} \left(2^{-\ell}|I(x,y)|\right)^{-1}\\
	&= \frac{2}{|I(x,y)|} \left[ 2e^{-2t|I(x,y)|^{-s}} + \sum_{\ell\geq1} 2^\ell e^{-2t|I(x,y)|^{-s}2^{\ell s}} \right]\\
	&= \frac{2}{|I(x,y)|}\ \eta_t\left(|I(x,y)|^{-s}\right)\\
	&= \frac{2}{\delta(x,y)}\ \eta_t\left(\frac{1}{\delta(x,y)^{s}}\right)\\
	&= \psi_t^2\big(\delta(x,y)\big).
	\end{align*}
	Item $(2)$ follows readily from the fact that for $0<r<\psi_{t}(+\infty)$ we have
	\begin{equation*}
		B_t(x,r) = \set{y\in\R^+\colon \psi_t(\delta(x,y))<r} = I,
	\end{equation*}
	where $I$ is the largest dyadic interval containing $x$ for which $\psi_{t}(\abs{I})$ is less than $r$.
\end{proof}

\bigskip

\medskip
\section*{Acknowledgment}
This work was carried out at IMAL in Santa Fe, supported by CONICET and UNL, and by project funds awarded by the Agencia I+D+i of MINCYT in Argentina.

\bigskip
\noindent{\footnotesize
	\noindent\textit{Affiliation:\,}
	\textsc{Instituto de Matem\'{a}tica Aplicada del Litoral, UNL, CONICET.}
	
	\smallskip
	\noindent\textit{Address:\,} \textmd{CCT CONICET Santa Fe, Predio ``Alberto Cassano'', Colectora Ruta Nac.~168 km 0, Paraje El Pozo, S3007ABA Santa Fe, Argentina.}
	
	\smallskip
	\noindent\textit{E-mail addresses:\,}
	\verb|mfacosta@santafe-conicet.gov.ar|;
	\verb|haimar@santafe-conicet.gov.ar|;\\ \verb|ivanagomez@santafe-conicet.gov.ar|;
	\verb|fmorana@santafe-conicet.gov.ar|
}

\end{document}